\documentclass[a4paper]{amsart}
\usepackage[margin={100pt, 100pt}, centering]{geometry}
\usepackage[utf8]{inputenc}
\usepackage[T1]{fontenc}
\usepackage{textcomp}
\usepackage{amsmath, amssymb}
\usepackage{amsthm}
\usepackage{mathrsfs}
\usepackage{esint}
\usepackage{float}
\usepackage{comment}
\usepackage{import}
\usepackage{xifthen}
\usepackage{pdfpages}
\usepackage{transparent}
\usepackage{hyperref}
\usepackage{subcaption}
\usepackage{caption}
\usepackage[most]{tcolorbox}
\usepackage[shortlabels]{enumitem}

\newtheorem{problem}{Problem}[section]
\newtheorem{theorem}[problem]{Theorem}
\newtheorem{proposition}[problem]{Proposition}
\newtheorem{corollary}[problem]{Corollary}
\newtheorem{conjecture}[problem]{Conjecture}
\newtheorem{question}[problem]{Question}

\newtheorem{lemma}[problem]{Lemma}
\newtheorem{remark}[problem]{Remark}

\theoremstyle{definition}
\newtheorem{definition}[problem]{Definition}

\newtcbox{\redbox}{colback=red!5!white,colframe=red!75!black, tcbox raise base}

\author{Xuan Yao}
\address{Department of Mathematics, Cornell University, Ithaca, NY, 14850}
\email{xy346@cornell.edu}

\title{A Note on Scalar curvature comparison rigidity for compact domains}

\begin{document}

\begin{abstract}
    We prove a generalization of Gromov's conjecture on scalar curvature rigidity of convex polytopes to arbitrary convex Riemannian polytope type domains via harmonic spinors on convex domians with boundary condition constructed by Brendle. In particular, we prove a rigidity result on comparison of scalar curvature and scaled mean curvature on the boundary for any convex domain in Euclidean space, which is a parallel of Shi-Tam's results.
\end{abstract}
\maketitle
\section{Introduction}
Comparison theorems of scalar curvature and mean curvature of a convex compact domain has been an interesting topic in differential geometry. Shi-Tam \cite{shi2002positive} proved a comparison theorem for $L^1$ norm of boundary mean curvature under the non-negative scalar curvature and the boundary isometric embedding assumptions. Gromov \cite{gromov2014dirac} conjectured the Riemannian polytope comparison theorems in high dimensions, which is true in dimension $2$ due to the Gauss-Bonnet Theorem. 
\begin{conjecture}[Gromov \cite{gromov2014dirac}]
    Let $P$ be a  convex $n$-dim polytope in $\mathbb R^n$, equip $P$ with Riemannian metric $g$ satisfying:
    \begin{enumerate}
        \item each face is mean convex w.r.t. $g$;
        \item the dihedral angles between paces w.r.t. $g$ is no larger than the diheral angles w.r.t. the Euclidean metric $g_0$;
        \item $R_g\geq 0$,
    \end{enumerate}
    then the metric $g$ must be flat.
\end{conjecture}
Much progress has been made on Gromov's conjecture. Gromov proved the cubical case and sketched a proof using Dirac operator and smoothout construction, see \cite{gromov2014dirac} \cite{gromov2021lecturesscalarcurvature} \cite{gromov2023convexpolytopesdihedralangles}. Using capilary minimal surfaces, Li \cite{li2020polyhedron} proved the conjecture for a large class of polytopes in dimension $3$, and he generalized his method to dimension $7$; see \cite{li2024dihedral}. Wang-Xie-Yu proposed developing spinor method on manifolds with corners \cite{wang2021gromov}. Brendle proved the Gromov's conjecture under the \textit{Matching angle Hypothesis} \cite{brendle2023matchingangle} by constructing a proper boundary condition for harmonic spinors and a smooth-out procedure. Brendle-Wang \cite{brendle2023gromovsrigiditytheorempolytopes} constructed another smooth-out procedure and proved the Gromov's conjecture under the acute angle comparison.

It is natural to consider the following question.
\begin{question}
    Can we generalize Gromov's conjecture to more general convex domians in $\mathbb R^n$?
\end{question}
We give affirmative answer to the question and we have the following results.
\begin{theorem}
    Suppose $\Omega^n$ is a convex Riemannian polytope type domain in $\mathbb{R}^n$, $n\geq 3$. Assume that $g$ is a Riemannian metric on $\Omega$ and it satisfies:
    \begin{enumerate}
        \item $R_g\geq 0$;
        \item $H_g^2g\geq H_0^2g_0$ for any regular point of $\partial \Omega$, and the metric tensors are restricted to $T\partial\Omega$;
        \item the matching angle hyphothesis is satisfied for any singular point of $\partial\Omega$.
    \end{enumerate}
    Then the metric $g$ must be Euclidean.
\end{theorem}
\begin{remark} $ $
    \begin{enumerate}
        \item Chai-Wang \cite{chai2023scalar} proved a rigidity comparison result for rotationally symmetric domains in dimension $3$ with different comparison assumptions;
        \item To guarantee the existence of harmonic spinors, we need the convexity assumption of the compact domain, we conjecture the convexity could be relaxed to weakly convex and mean convex.
    \end{enumerate}
\end{remark}

As a corollary, we have
\begin{corollary}
    Suppose $\Omega^n$ is a smoooth convex domain in $\mathbb R^n$, with dimension $n\geq 3$. Assume that $g$ is a Riemannian metric on $\Omega$ and it satisfies:
    \begin{enumerate}
        \item $R_g\geq 0$;
        \item $H_g^2g\geq H_0^2g_0$ on $\partial \Omega$, where we restrict the metric tensors to $T\partial\Omega$;
    \end{enumerate}
    then the metric $g$ must be Euclidean.
\end{corollary}

Another natural question is 
\begin{question}
    Do we have a corresponding exterior result which serves as a variant of the Positive Mass Theorem?
\end{question}
We give an affirmative answer to this question when the domain is $\mathbb R^3\setminus B_r(0)$; see Section \ref{sec: further discussion} for more details.

\subsection{Outline of the proof}
We begin by introducing the prelimaries on spin geometry and Brendle's construction of boundary condition for harmonic spinors in Section \ref{sec:pre}. 

The key ingredient of the proof is Morrey norm estimates of $\max\{-(H_g-\|dN\|_{tr}),0\}$. As in \cite{brendle2023matchingangle}, we divide the approximation surface into three parts: near face parts (co-dim $1$), near edge parts (co-dim 2) and near vertex parts (co-dim 3), the estimates follow from the (mean curvature)-metric comparison and angle conditions, and the proof is complete, see Section \ref{sec:morrey norm}. 

We discuss the exterior version of Gromov's conjecture in Section \ref{sec: further discussion} and prove a special case when the convex domain is a stantand ball in Euclidean space via Inverse Mean Curvature Flow.

\subsection{Acknowledgements}
I would like to thank Prof. Xin Zhou and Prof. Xianzhe Dai for their helpful discussions on this topic. I am also grateful to Dongyeong Ko and Liam Mazurowski for their helpful discussions on this project. X.Y. is supported by Hutchinson Fellowship.
\section{Prelimaries}\label{sec:pre}

\subsection{Definitions and Notations}
We first introduce the definition of {\em Riemannian polytope type domain}.
\begin{definition}
    Suppose $\Omega$ is a bounded open domain of a smooth manifold $M$, we 
    say that $\Omega$ is a {\em Riemannian polytope type domain} if it satisfies:
    \begin{enumerate}
        \item $\Omega$ is smooth in the interior and $\Omega$ is not an empty set;
        \item $\Omega=\cap_{i\in I}\{x\in M:u_i(x)\leq 0\}$, $\partial\Omega=\cap_{i\in I}\{x\in M: u_i(x)=0\}$, where $u_i:M\to \mathbb R$ is 
              a smooth function and $0$ is a regular value of $u_i$, for each $i\in I$, $|I|=k$ is finite.
    \end{enumerate}
\end{definition}

\begin{remark}
    If $u_i:\mathbb R^n\to\mathbb R$ are linear functions for each $i\in I$,
    then $\Omega$ is a standard polytope in $\mathbb{R}^n$.
\end{remark}

\begin{remark}
    If $k=1$ then $\Omega$ is a smooth open domain.
\end{remark}

\begin{definition}[Matching Angle Hypothesis] 
We say a Riemannian polytope type domain $\Omega$ satisfies the {\em matching angle hypothesis} for any singular point $x\in \partial\Omega$, if for any $x\in \partial\Omega$ such that $u_i(x)=u_j(x)=0$, $i,j\in I$, then we have
    $g_0(N_i,N_j)=g(\nu_i,\nu_j)$ at $x$. Here, $N_{l}=\frac{\nabla^{g_0}u_l}{|\nabla^{g_0}u_l|}$, $\nu_r=\frac{\nabla^g u_l}{|\nabla^g u_l|}$, $l=i,j$.
\end{definition}
\subsection{Spinor method}
In this subsection, we recall Brendle's construction of boundary condition for
harmonic spinor equation, which serves as a key ingredient in the resolution 
of Gromov's conjecture under {\em the Matching angle hyphothesis}, see Brendle \cite{brendle2023matchingangle}.

We assume that $\Omega\subset \mathbb{R}^n$ is a convex bounded open domain in $\mathbb{R}^n$ and $n\geq 3$ is an odd integer throughout this section.

Let $\{E_1,E_2,\dots, E_n\}$ be the orthonormal basis of $\mathbb{R}^n$, $Cl(n,\mathbb C)$ be the Clifford algebra. Recall that $Cl(n,\mathbb C)$ is generated by $1, E_1,\dots, E_n$ via the relation:
\[
E_iE_j+E_jE_i=-2\delta_{ij}.
\]

Denote the complex vector space of dimension $m=[\frac{n}{2}]$ as $\Delta_n$, which is equipped with a Hermitian inner product. The Spin representation gives a surjective algebra homomorphism $\hat{\rho}: Cl(n,\mathbb C)\to End(\Delta_n)$. Suppose we fix an orthornormal basis of $\Delta_n$, $\{\hat{s}_1,\hat{s}_2,\dots,\hat{s}_m\}$, then we define:
\[
\omega_{a\alpha\beta}:=\langle \rho(E_a)\hat{s}_{\alpha},\hat{s}_{\beta}\rangle,
\]
for $a\in\{1,2\dots, n\}$, $\alpha,\beta\in\{1,2,\dots,m\}$. We will omit the $\hat{\rho}$ for simplification.

For a general metric $g$ on $\Omega$, we use $\{e_1,e_2\dots,e_n\}$ to denote a local orthonormal frame, $\mathcal{S}$ to denote the spinor bundle on $\Omega$. Recall the definition of the Dirac operator is 
\[
\mathcal D s:=\sum_{i=1}^n e_i\cdot \nabla_{e_i}s,\quad \forall s\in \mathcal{S}.
\]

Denote $\partial\Omega$ as $\Sigma$. Suppose $\{e_1,\dots, e_{n-1}\}$ is a local orthonormal frame of $\Sigma$, and $\nu$ is the boundary unit normal, then the boundary Dirac operator is defined as
\[
\mathcal{D}^{\Sigma} s:=\sum_{i=1}^{n-1}\nu\cdot e_i\cdot \nabla_{e_i}s+\frac{1}{2}H s,\quad \forall s\in \mathcal{S}|_{\Sigma}.
\]

Given a map $N:\Sigma\to S^{n-1}$, we define a bundle map $\chi: \mathcal{E}|_{\Sigma}\to\mathcal{E}|_{\Sigma}$ as
\[
(\chi s)_{\alpha}=-\sum_{a=1}^n\sum_{\beta=1}^m\langle N,E_a\rangle \omega_{a\alpha\beta}\nu\cdot s_{\beta},
\]
where $\mathcal{E}=\bigoplus_{m}\mathcal{S}$ is a complex bundle over $\Omega$.
\begin{proposition}[Brendle \cite{brendle2023matchingangle}]
    Suppose that $s=(s_1,\dots,s_m)\in \mathcal{E}$ is an $m$-tuple of spinors, then we have
    \begin{align*}
        &-\int_{\Omega}\sum_{\alpha=1}^m|\mathcal{D}s_{\alpha}|^2dV_g+\int_{\Omega}\sum_{\alpha=1}^m|\nabla s_{\alpha}|^2dV_g+\frac{1}{4}\int_{\Omega}\sum_{\alpha=1}^mR_g|s_{\alpha}|^2dV_g\\
        &\leq \frac{1}{2}\int_{\Sigma}\sum_{\alpha=1}^m\langle \mathcal{D}^{\Sigma}s_{\alpha},s_{\alpha}-(\chi s)_{\alpha}\rangle da_g+\frac{1}{2}\int_{\Sigma}\sum_{\alpha=1}^m\langle s_{\alpha}-(\chi s)_{\alpha},\mathcal{D}^{\Sigma}s_{\alpha}\rangle da_g\\
        &-\frac{1}{2}\int_{\Sigma}(H-\|dN\|_{tr})\left(\sum_{\alpha=1}^m |s_{\alpha}|^2\right)da_g.
    \end{align*}
\end{proposition}

Brendle  also proved the existence of the $m$-tuple of harmonic spinors with boundary condition $\chi s=s$.

\begin{proposition}[Brendle \cite{brendle2023matchingangle}]\label{prop:solvability}
    Suppose $\Omega\subset \mathbb{R}^n$ is a bounded covex smooth open domain, and $n\geq 3$ is an odd integer. Assume further that $N:\Sigma\to S^{n-1}$ is homotopic to the Euclidean Gauss map of $\Sigma=\partial\Omega$. Then the operator
    \[
    H^1(\Omega,\mathcal{E})\to L^2(\Omega,\mathcal{E})\bigoplus H^{\frac{1}{2}}(\Sigma,\mathcal{F}^-),\quad s\to (\mathcal{D}s,s-\chi s)
    \]
    is Fredholm and has Fredholm index at least $1$, where $\mathcal{F}^{-}:=\text{Ker}(\text{id}+\chi)$.
\end{proposition}

\begin{remark}
    The operator being Fredholm only requires that $\Omega$ is smooth and bounded, but to have the Fredholm index lower bound, one has to require that $\Omega$ is convex and $N$ is homotopic to the Euclidean Gauss map.
\end{remark}

\begin{lemma}\label{lem:comparison}
    Suppose $x\in \partial\Omega$ is a regular point of $\partial\Omega$, and $u_i(x)=0$, we specify $N_i:T\partial\Omega\cap \{u_i(x)=0\}\to S^{n-1}$ as the Euclidean Gauss map of the zero level set of $u_i$, then we have 
    \[
    \|dN_i(x)\|_{tr}\leq H_g(x).
    \]
\end{lemma}
\begin{proof}
    Suppose $\tilde{\lambda}_j$ are the singular values of $dN_i$ w.r.t. the $g$ metric, then 
    \[
    dN_i(\tilde{e}_j)=\tilde{\lambda}_jE_j,\quad \j\in\{1,2,\dots,n-1\}
    \]
    where $\{\tilde{e}_j\}_{j=1}^{n-1}$ is an o.n. basis of $T_x\partial\Omega$ w.r.t. $g$, $\{E_j\}_{j=1}^{n-1}$ is an o.n. basis of $TS^{n-1}$ w.r.t. the round metric $g_{r}$.

    Suppose $\lambda_j$ are non-negative singular values of $dN_i$ w.r.t. the $g_0$, then
    \[
    dN_i(e_j)=\lambda_j F_j,\quad j\in\{1,2,\dots, n-1\}
    \]
    where $\{e_j\}_{j=1}^{n-1}$ is an o.n. basis of $T_x\partial\Omega$ w.r.t. $g$, $\{F_j\}_{j=1}^{n-1}$ is an o.n. basis of $TS^{n-1}$ w.r.t. the round metric $g_{r}$.

    By deinifion, we compute 
    \begin{align*}
        \|dN_i(x)\|_{tr}=&\sum_{j=1}^{n-1}\tilde{\lambda}_j=\sum_{j=1}^{n-1}g_r(dN_i(\tilde{e}_j),E_j)\\
                        =&\sum_{j=1,l=1}^{n-1}g_0(e_l,\tilde{e}_j)g_r(dN_i(e_l),E_j)=\sum_{j,l=1}^{n-1}\lambda_lg_0(e_l,\tilde{e}_j)g_r(F_l,E_j)\\
                        =&\sum_{j,l=1}^{n-1}\lambda_lg_0(e_l,C_{l,j}\tilde{e}_j)=\sum_{l=1}^{n-1}\lambda_lg_0(e_l,\bar{e}_j)\\
                        \leq& \sum_{l=1}^{n-1}\lambda_l|\bar{e}_l|_{g_0}.
    \end{align*}
    Where $(C_{l,j})_{(n-1)\times (n-1)}\in O(n-1)$, $\{\bar{e}_j\}_{j=1}^{n-1}=\{C_{j,l}\tilde{e}_l\}_{j=1}^{n-1}$ is another o.n. basis of $T_x\partial\Omega$ w.r.t. the metric $g$.
    
    Since 
    \[
    H_g^2g\geq H_{0}^2g_0,
    \]
    we have that
    \[
    |\bar{e}_l|_{g_0}\leq \frac{H_g}{H_{g_0}},
    \]
    thus we obtain
    \begin{align*}
        \|dN_i(x)\|_{tr}&\leq \sum_{l=1}^{n-1}\lambda_l\frac{H_g}{H_{g_0}}\\
                        &=H_g.
    \end{align*}
\end{proof}


\begin{corollary}
    Suppose $\Omega$ is a convex domain in $\mathbb{R}^n$, where $n\geq 3$ is an odd integer. Equip $\Omega$ with a Riemannian metric $g$ such that $R_g\geq 0$ for any $x\in\Omega$, $H_g^2(x)g\geq H_0^2(x)g_0$ for any $x\in\partial\Omega$, then $g$ must be Euclidean.
\end{corollary}
\begin{remark}
    For the general dimension $n\geq 3$, we need to prove the \textit{Riemannian polytope} version first.
\end{remark}

\section{Morrey norm estimates}\label{sec:morrey norm}
The main difficulty of Morrey norm estimates in our setting compared with the standard convex polytope is that the unit normal vector of each {\em face} of the {\em Riemannian polytope type domain} is changing in the Euclidean space.

The key ingredients of the Morrey norm estimates are dividing the smooth-out approximation surface into 3 parts: near face(co-dim $1$) parts, near edge(co-dim $2$) parts and near vertex(co-dim $3$) parts. Both near edge parts and near vertex parts are dominated by the angle assumption, only the near face parts are dominated by the mean curvature-metric comparison. Therefore, in our setting, we only need to focus on the near face part estimates and area estimates of other parts.

Since $\Omega$ has a non-empty interior, there exists a positive number $\lambda_0$, such that for $\lambda\geq\lambda_0$,
\[
\Omega_{\lambda}:=\{x\in\mathbb R^n: \sum_{i=1}^ke^{\lambda u_i(x)}\leq 1\}\neq\emptyset.
\]
We denote $\partial\Omega_{\lambda}$ as $\Sigma_{\lambda}$. For each $\lambda\geq\lambda_0$, 
$\Omega_{\lambda}$ is a compact convex domain with smooth boundary $\Sigma_{\lambda}$.

\begin{lemma}
    If $\lambda$ is sufficiently large, then 
    \[
    \inf_{\Sigma_{\lambda}}\left\vert \sum_{i=1}^ke^{\lambda u_i}du_i\right\vert\geq C_1^{-1},
    \]
    and
    \[
    \inf_{\Sigma_{\lambda}}\left\vert\sum_{i=1}^ke^{\lambda u_i}|\nabla u_i|N_i\right\vert_{g_0}\geq C_2^{-1},
    \]
    where $C_1,C_2,$ are positive constants independent of $\lambda$.
\end{lemma}
\begin{proof}
    The proof is exactly the same as \textbf{Lemma 3.5} and \textbf{Lemma 3.6} in Brendle \cite{brendle2023matchingangle}. 
\end{proof}
\begin{remark}
    Brendle's proof of \text{Lemma 3.5, 3.6} does not use the fact that $u_i$ is a linear function, therefore, the same results follow in our setting.
\end{remark}

The unit normal vector field of $\Sigma_{\lambda}$ w.r.t. $g$ is
\[
\nu(x):=\frac{\sum_{i=1}^ke^{\lambda u_i}|\nabla u_i|\nu_i}{\left\vert\sum_{i=1}^ke^{\lambda u_i}|\nabla u_i|\nu_i\right\vert},
\]
and a map $N:\Sigma_{\lambda}\to S^{n-1}$ 
\[
N(x):=\frac{\sum_{i=1}^ke^{\lambda u_i}|\nabla u_i|N_i}{\left\vert\sum_{i=1}^ke^{\lambda u_i}|\nabla u_i|N_i\right\vert_{g_0}}.
\]
We use $\vert\cdot\vert$ to denote the norm w.r.t. the $g$ metric, $\vert\cdot\vert_{g_0}$ to denote the norm w.r.t. the Euclidean metric.

\begin{proposition}
    For any $x\in \Sigma_{\lambda}$, we have that
    \[
    H-\|dN\|_{tr}\geq V_{\lambda},
    \]
    where $V_{\lambda}:\Sigma_x\to\mathbb R$.
    \begin{align*}
    V_{\lambda}&=\lambda\frac{\sum_{i=1}^ke^{\lambda u_i}|\nabla u_i|^2|\pi(\nu_i)|^2}{\left\vert\sum_{i=1}^ke^{\lambda u_i}|\nabla u_i|\nu_i\right\vert}+\frac{\sum_{i=1}^ke^{\lambda u_i}\left(\Delta u_i-\nabla^2u_i(\nu,\nu)\right)}{\left\vert\sum_{i=1}^ke^{\lambda u_i}|\nabla u_i|\nu_i\right\vert}\\
    &-\lambda\frac{\sum_{i=1}^ke^{\lambda u_i}|\nabla u_i|^2|\pi(\nu_i)||P(N_i)|}{\left\vert\sum_{i=1}^ke^{\lambda u_i}|\nabla u_i|N_i\right\vert_{g_0}}-\frac{\sum_{i=1}^ke^{\lambda u_i}|\nabla|\nabla u_i||P(N_i)|}{\left\vert\sum_{i=1}^ke^{\lambda u_i}|\nabla u_i|N_i\right\vert_{g_0}}\\
    &-\frac{\sum_{i=1}^ke^{\lambda u_i}|\nabla u_i|\|dN_i\|_{tr}}{\left\vert\sum_{i=1}^ke^{\lambda u_i}|\nabla u_i|N_i\right\vert_{g_0}}.
    \end{align*}
    $\pi:T_x\Omega_{\lambda}\to T_x\Sigma_{\lambda}$ is the orthogonal projection w.r.t. metric $g$, $P:T_x\Omega_{\lambda}\to T_x\Sigma_{\lambda}$ is the orthogonal projection w.r.t. Euclidean metric.
\end{proposition}
\begin{proof}
    The mean curvature of $\Sigma_{\lambda}$ w.r.t the $g$ metric by definition is
    \begin{align*}
        H&=\nabla^{\Sigma_{\lambda}}\cdot\left(\frac{\sum_{i=1}^ke^{\lambda u_i}|\nabla u_i|\nu_i}{\left\vert\sum_{i=1}^ke^{\lambda u_i}|\nabla u_i|\nu_i\right\vert}\right)\\
        &=\lambda\frac{\sum_{m=1}^{n-1}\sum_{i=1}^ke^{\lambda u_i}\langle \nabla u_i,e_m\rangle^2}{\left\vert\sum_{i=1}^ke^{\lambda u_i}|\nabla u_i|\nu_i\right\vert}+\frac{\sum_{i=1}^ke^{\lambda u_i}\left(\Delta u_i-\nabla^2u_i(\nu,\nu)\right)}{\left\vert\sum_{i=1}^ke^{\lambda u_i}|\nabla u_i|\nu_i\right\vert}\\
        &=\lambda\frac{\sum_{i=1}^ke^{\lambda u_i}|\nabla u_i|^2|\pi(\nu_i)|^2}{\left\vert\sum_{i=1}^ke^{\lambda u_i}|\nabla u_i|\nu_i\right\vert}+\frac{\sum_{i=1}^ke^{\lambda u_i}\left(\Delta u_i-\nabla^2u_i(\nu,\nu)\right)}{\left\vert\sum_{i=1}^ke^{\lambda u_i}|\nabla u_i|\nu_i\right\vert}.
    \end{align*}

    For any $\zeta\in T_x\Sigma_{\lambda}$, we have that
    \begin{align*}
        dN(\zeta)&=\lambda \frac{\sum_{i=1}^ke^{\lambda u_i}|\nabla u_i|\langle \nabla u_i,\zeta\rangle P(N_i)}{\left\vert\sum_{i=1}^ke^{\lambda u_i}|\nabla u_i|N_i\right\vert}+\frac{\sum_{i=1}^ke^{\lambda u_i}\langle \nabla |\nabla u_i|,\zeta\rangle P(N_i)}{\left\vert\sum_{i=1}^ke^{\lambda u_i}|\nabla u_i|N_i\right\vert}\\
        &+\frac{\sum_{i=1}^ke^{\lambda u_i}|\nabla u_i|P(dN_i(\zeta))}{\left\vert\sum_{i=1}^ke^{\lambda u_i}|\nabla u_i|N_i\right\vert}.
    \end{align*}
    Since the trace norm satisfies the triangle inequality, we have that
    \begin{align*}
        \|dN\|_{tr}&\leq \lambda\frac{\sum_{i=1}^ke^{\lambda u_i}|\nabla u_i|^2|\pi(\nu_i)||P(N_i)|}{\left\vert\sum_{i=1}^ke^{\lambda u_i}|\nabla u_i|N_i\right\vert_{g_0}}+\frac{\sum_{i=1}^ke^{\lambda u_i}|\nabla|\nabla u_i||P(N_i)|}{\left\vert\sum_{i=1}^ke^{\lambda u_i}|\nabla u_i|N_i\right\vert_{g_0}}\\
        &+\frac{\sum_{i=1}^ke^{\lambda u_i}|\nabla u_i|\|dN_i\|_{tr}}{\left\vert\sum_{i=1}^ke^{\lambda u_i}|\nabla u_i|N_i\right\vert_{g_0}}.
    \end{align*}

    Then the desired inequality follows by subtracting these two inequalities.
\end{proof}

\begin{lemma}\label{lem:oestimate}
As $\lambda\to\infty$, we have
    \[
    \sup_{\Sigma_{\lambda}}\max\{-V_{\lambda},0\}\leq o(\lambda). 
    \]   
\end{lemma}
\begin{proof}
    The proof is the same as \textbf{Proposition 3.10} in Brendle \cite{brendle2023matchingangle}.
\end{proof}
\begin{remark}
    To prove Lemma \ref{lem:oestimate}, we only need to estimate the following two terms
    \[
    \lambda\frac{\sum_{i=1}^ke^{\lambda u_i}|\nabla u_i|^2|\pi(\nu_i)|^2}{\left\vert\sum_{i=1}^ke^{\lambda u_i}|\nabla u_i|\nu_i\right\vert}-\lambda\frac{\sum_{i=1}^ke^{\lambda u_i}|\nabla u_i|^2|\pi(\nu_i)||P(N_i)|}{\left\vert\sum_{i=1}^ke^{\lambda u_i}|\nabla u_i|N_i\right\vert_{g_0}}
    \]
    since other terms are bounded by a constant independent of $\lambda$. 
\end{remark}

When $k=1$, the results is obviously true by Lemma \ref{lem:comparison}.

We first consider the case when $k\geq 3$.

\begin{definition}
    If $k\geq 3$, for three distinct elements $i_1,i_2,i_3\in I=\{1,2,\dots,k\}$, we denote $G_{\lambda}^{(i_1,i_2,i_3)}$ as the subsets of $\Sigma_{\lambda}$ where $u_{i_1}\geq u_{i_2}\geq u_{i_3}$, and $u_{i_3}\geq u_{i}$ for $i\in I\setminus \{i_1,i_2,i_3\}$.
\end{definition}

\begin{lemma}
    For each $i\in I$, and a fixed exponent $\sigma\in [1,3/2)$, and let $B_r(p)$ denote a Euclidean ball of radius $0<r\leq 1$. If $\lambda r$ is sufficiently large then 
    \[
    \left(r^{\sigma+1-n}\int_{G_{\lambda}^{(i_1,i_2,i_3)}\cap\{u_{i_2}\leq -\lambda^{-\frac{7}{8}}r^{\frac{1}{8}}\}\cap B_r(p)}(\max\{-V_{\lambda},0\})^{\sigma}\right)^{1/\sigma}\leq C\lambda re^{-(\lambda r)^{\frac{1}{8}}},
    \]
    where $C$ is independent of $\lambda$, $p$ is any point in Euclidean space.
\end{lemma}
\begin{proof}
    For any $i\in I\setminus\{i_1\}$, $e^{\lambda u_i(x)}\leq e^{-(\lambda r)^{\frac{1}{8}}}$ for any $x\in G_{\lambda}^{(i_1,i_2,i_3)}$. 
    
    Note that
    \[
    |\nu-\nu_{i_1}|\leq Ce^{-(\lambda r)^{\frac{1}{8}}},\quad |N-N_{i_1}|_{g_0}\leq Ce^{-(\lambda r)^{\frac{1}{8}}},
    \]
    where $C$ is a constant independent of $\lambda$.

    We can then conclude that the absolute value of the following three terms 
    \begin{align*}
        &\lambda \frac{\sum_{i\in I}e^{\lambda u_i}|\nabla u_i|^2|\pi(\nu_i)|^2}{\left\vert\sum_{i\in I}e^{\lambda u_i}|\nabla u_i|\nu_i\right\vert}-\lambda \frac{\sum_{i\in I}e^{\lambda u_i}|\nabla u_i|^2|\pi(\nu_i)||P(N_i)|}{\left\vert\sum_{i\in I}e^{\lambda u_i}|\nabla u_i|N_i\right\vert_{g_0}}\\
        &-\frac{\sum_{i\in I}e^{\lambda u_i}|\nabla|\nabla u_i|||P(N_i)|}{\left\vert\sum_{i\in I}e^{\lambda u_i}|\nabla u_i|N_i\right\vert_{g_0}}
    \end{align*}
    is bounded from above by 
    \[
    C\lambda e^{-(\lambda r)^{\frac{1}{8}}}.
    \]

    Moreover, since $u_{i_1}(x)\geq -C\lambda^{-1}e^{-(\lambda r)^{\frac{1}{8}}}$, there exists a $y\in\mathbb{R}^n$, such that $u_{i_1}(y)=0$ and $d_{eucl}(x,y)\leq C\lambda^{-1}e^{-(\lambda r)^{\frac{1}{8}}}$. For $\lambda r$ sufficiently large, we have
    \[
    u_i(y)\leq u_i(x)+C\lambda^{-1}e^{-(\lambda r)^{\frac{1}{8}}}\leq -\lambda^{-\frac{7}{8}}r^{\frac{1}{8}}+C\lambda^{-1}e^{-(\lambda r)^{\frac{1}{8}}}\leq 0,\quad \forall i\in I\setminus\{i_1\}.
    \]
    We then have $y\in\partial\Omega$, and $u_{i_1}(y)=0$. 

    By Lemma \ref{lem:comparison}, we have 
    \begin{align*}
        \frac{\Delta u_{i_1}-\nabla^2u_{i_1}(\nu_{i_1},\nu_{i_1})}{|\nabla u_{i_1}|}(y)- \|dN_{i_1}(y)\|_{tr}\geq 0.
    \end{align*}
    We then obtain that at the point $y$
    \begin{align*}
        \Delta u_{i_1}-\nabla^2u_{i_1}(\nu_{i_1},\nu_{i_1})-|\nabla u_{i_1}|\|dN_{i_1}\|_{tr}\geq 0.
    \end{align*}
    Which implies that at the point $x$ we have
    \[
    \Delta u_{i_1}-\nabla^2u_{i_1}(\nu_{i_1},\nu_{i_1})-|\nabla u_{i_1}|\|dN_{i_1}\|_{tr}\geq -C\lambda^{-1}e^{-(\lambda r)^{\frac{1}{8}}}.
    \] 
    We then compute that
    \begin{align*}
        &\frac{\sum_{i\in I}e^{\lambda u_i}(\Delta u_i-\nabla^2u_i(\nu,\nu))}{\left\vert\sum_{i\in I}e^{\lambda u_i}|\nabla u_i|\nu_i\right\vert}-\frac{\sum_{i\in I}e^{\lambda u_i}|\nabla u_i|\|dN_i\|_{tr}}{\left\vert\sum_{i\in I}e^{\lambda u_i}|\nabla u_i|N_i\right\vert_{g_0}}\\
        &\geq \frac{e^{\lambda u_{i_1}}(\Delta u_{i_1}-\nabla^2 u_{i_1}(\nu_{i_1},\nu_{i_1}))}{\left\vert\sum_{i\in I}e^{\lambda u_i}|\nabla u_i|\nu_i\right\vert}-\frac{e^{\lambda u_{i_1}}|\nabla u_{i_1}|\|dN_{i_1}\|_{tr}}{\left\vert \sum_{i\in I}e^{\lambda u_i}|\nabla u_i|N_i\right\vert_{g_0}}-C e^{(\lambda r)^{-\frac{1}{8}}}\\
        &\geq C\left(\left\vert\sum_{i\in I}  e^{\lambda u_i}|\nabla u_i|N_i\right\vert_{g_0}-\left\vert\sum_{i\in I}e^{\lambda u_i}|\nabla u_i|\nu_i\right\vert\right)-Ce^{-(\lambda r)^{-\frac{1}{8}}}\\
        &\geq -Ce^{-(\lambda r)^{\frac{1}{8}}},
    \end{align*}
    Where $C$ is a positive constant independent of $\lambda$.
    
    At the point $x$, we have the following estimate
    \begin{align*}
        V_{\lambda}(x)\geq -C\lambda e^{-(\lambda r)^{-\frac{1}{8}}}-Ce^{-(\lambda r)^{\frac{1}{8}}}
    \end{align*}
    The Morrey norm of the negative part of $V_{\lambda}$ follows.
\end{proof}
\begin{remark}
    Since $\Sigma_{\lambda}$ bounds a convex domain in Euclidean space, $\Sigma_{\lambda}\cap B_r(p)$ has area at most $C r^{n-1}$, where $C$ only depends on $n$.
\end{remark}

\begin{lemma}
    For each $i\in I$, and a fixed exponent $\sigma\in [1,3/2)$, and let $B_r(p)$ denote a Euclidean ball of radius $0<r\leq 1$. If $\lambda r$ is sufficiently large then 
    \[
    \left(r^{\sigma+1-n}\int_{G_{\lambda}^{(i_1,i_2,i_3)}\cap\{u_{i_2}\geq -\lambda^{-\frac{7}{8}}r^{\frac{1}{8}}\}\cap\{u_{i_3}\leq -\lambda^{-\frac{3}{4}}r^{\frac{1}{4}}\}\cap B_r(p)}(\max\{-V_{\lambda},0\})^{\sigma}\right)^{1/\sigma}\leq C(\lambda r)^{\frac{1}{8}-\frac{7}{8\sigma}}
    \]
    where $C$ is independent of $\lambda$, $p$ is any point in Euclidean space.
\end{lemma}
\begin{proof}
    The proof is the same as the proof of \textbf{Lemma 3.14} in \cite{brendle2023matchingangle}.
\end{proof}
\begin{lemma}
    For each $i\in I$, and a fixed exponent $\sigma\in [1,3/2)$, and let $B_r(p)$ denote a Euclidean ball of radius $0<r\leq 1$. If $\lambda r$ is sufficiently large then 
    \[
    \left(r^{\sigma+1-n}\int_{G_{\lambda}^{(i_1,i_2,i_3)}\cap\{u_{i_3}\geq -\lambda^{-\frac{3}{4}}r^{\frac{1}{4}}\}\cap B_r(p)}(\max\{-V_{\lambda},0\})^{\sigma}\right)^{1/\sigma}\leq C(\lambda r)^{1-\frac{3}{2\sigma}}
    \]
    where $C$ is independent of $\lambda$, $p$ is any point in Euclidean space.
\end{lemma}
\begin{proof}
     The proof is the same as the proof of \textbf{Lemma 3.14} in \cite{brendle2023matchingangle}.
\end{proof}

\begin{remark}
    The proofs of the above two lemmas do not need the fact that $u_i$ is linear, the extra terms in $V_{\lambda}$ compared to that term in \cite{brendle2023matchingangle} are bounded, the proofs then follow from the area estimate of the integrated area.
\end{remark}

When $k=2$, the above estimates also hold.

Combining the above Lemmas we obtain
\begin{corollary}
Let us fix an exponent $\sigma\in[1,\frac{3}{2})$, we have
\[
\sup_{p\in\mathbb{R}^n}\sup_{0<r\leq 1}\left(r^{\sigma+1-n}\int_{\Sigma_{\lambda}\cap B_r(p)}(\max\{-V_{\lambda},0\})^{\sigma}\right)^{\frac{1}{\sigma}}\to 0
\]
as $\lambda\to\infty$.
\end{corollary}

\begin{proof}[Proof of the Main Theorem] $ $
\begin{enumerate}
    \item \textit{We first prove the case when $n\geq 3$ is odd.}
    
With the Morrey Norm estimates, we apply Brendle's arguments in \textbf{Section 4} \cite{brendle2023matchingangle}, the proof is complete.
\item \textit{$n\geq 3$ is even}

$\Omega\times[0,1]$ is a \textit{Riemannian polytope type domain} of dimension $n+1$. Equip $\Omega\times[0,1]$ with metric $g+dt^2$, then both the \textit{Matching Angle hypothesis} and $H^2_{g+dt^2}(g+dt^2)\geq H_0^2g_0$ are satisfied, apply the odd dimensional result, the proof is completed.
\end{enumerate}

\end{proof}

\section{Further discussions}\label{sec: further discussion}
It is natural to consider an exterior generalization of Gromov's conjecture.
\begin{conjecture}\label{conj:non-compact}
    Suppose $\Omega$ is a convex smooth domain in $\mathbb R^n$, $n\geq 3$. Let $(\mathbb R^n\setminus \Omega,g)$ be an asymptotically flat Riemannian manifold, and $H_g^2g\leq H_0^2g_0$, then the ADM mass of $(\mathbb R^n\setminus \Omega,g)$ is non-negative, and it is zero if and only if $g$ is Euclidean.
\end{conjecture}

\begin{remark}
    This conjecture should also extend to \textit{Riemannian polytope type domain}, and the Morrey norm estimates are not very different from Section \ref{sec:morrey norm}. The tricky part is the non-compact version of Proposition \ref{prop:solvability}.
\end{remark}

We prove a special case of Conjecture \ref{conj:non-compact}.

\begin{proposition}\label{prop:special case}
    Suppose $(\mathbb R^3\setminus B_r(0),g)$ is asymptotically flat, and $H_g^2g\leq H_0^2g_0$, $H_g>0$, then $m_{ADM}(g)\geq 0$, and the equality is achieved if and only if $g$ is Euclidean.
\end{proposition}

To prove the special case, we need to use the Hawking mass monotonicity along inverse mean curvature flow.

\begin{theorem}[Huisken-Ilmanen \cite{huisken2001inverse}]\label{thm:imcf}
    Suppose $(N,h)$ is an asymptotically flat $3$-manifold with non-negative scalar curvature. If $E_0$ is an open precompact minimizing hull in $N$ having a smooth boundary $\partial E_0$, then there is a precompact locally lipschitz $\phi$ satisfying:
    \begin{enumerate}
        \item For $t\geq 0$, $\Sigma_t=\partial\{\phi<t\}$ defines an increasing family of $C^{1,\alpha}$ surfaces such that $\Sigma_0=\partial E_0$;
        \item For almost all $t\geq 0$, the mean curvature of $\Sigma_t$ is $|\nabla \phi|_{\Sigma_t}$ in the distributional sense;
        \item The Hawking mass
        \[
        m_H(\Sigma_t):=\sqrt{\frac{|\Sigma_t|}{(16\pi)^3}}\left(16\pi-\int_{\Sigma_t}H^2d\sigma_h\right)
        \]
        is a nondecreasing function of $t\geq 0$ provided the Euler characteristic $\chi(\Sigma_t)\leq 2$ for all $t\geq 0$;
        \item $\lim_{t\to\infty}m_H(\Sigma_t)\leq m_{ADM}(N)$.
    \end{enumerate}
\end{theorem}
\begin{proof}[Proof of Proposition \ref{conj:non-compact}]
    Denote $\mathbb R^3\setminus B_r(0)$ as $M$, since $\partial M$ is mean convex w.r.t. $g$, we can run the IMCF starting at $\partial M$, by Theorem \ref{thm:imcf}, we obtain
    \begin{align*}
        m_{ADM}(g)&\geq \lim_{t\to\infty}m_H(\Sigma_t)\\
               &\geq m_H(\Sigma_0)\\
               &=\sqrt{\frac{|\partial M|}{(16\pi)^3}}\left(16\pi-\int_{\partial M}H^2d\sigma_g\right).
    \end{align*}
    Since $H_g^2g\leq H_0^2g_0$ on $\partial M$, we have
    \[
    \int_{\partial M}H_g^2d\sigma_g\leq \int_{\partial M}H_0^2d\sigma_{g_0}=16\pi
    \]
    and thus 
    \[
    m_{ADM}(g)\geq 0.
    \]
    The equality is obtained if and only if the metric $g$ is Euclidean.
\end{proof}

\bibliographystyle{plain}
\bibliography{reference}

\begin{thebibliography}{10}

\bibitem{brendle2023gromovsrigiditytheorempolytopes}
S.~Brendle and Y.~Wang.
\newblock On gromov's rigidity theorem for polytopes with acute angles, 2023.

\bibitem{brendle2023matchingangle}
Simon Brendle.
\newblock Scalar curvature rigidity of convex polytopes.
\newblock {\em Inventiones mathematicae}, 235(2):669--708, 2024.

\bibitem{chai2023scalar}
Xiaoxiang Chai and Gaoming Wang.
\newblock Scalar curvature comparison of rotationally symmetric sets.
\newblock {\em arXiv preprint arXiv:2304.13152}, 2023.

\bibitem{gromov2014dirac}
Misha Gromov.
\newblock Dirac and plateau billiards in domains with corners.
\newblock {\em Central European Journal of Mathematics}, 12:1109--1156, 2014.

\bibitem{gromov2021lecturesscalarcurvature}
Misha Gromov.
\newblock Four lectures on scalar curvature, 2021.

\bibitem{gromov2023convexpolytopesdihedralangles}
Misha Gromov.
\newblock Convex polytopes, dihedral angles, mean curvature and scalar
  curvature, 2023.

\bibitem{huisken2001inverse}
Gerhard Huisken and Tom Ilmanen.
\newblock The inverse mean curvature flow and the riemannian penrose
  inequality.
\newblock {\em Journal of Differential Geometry}, 59(3):353--437, 2001.

\bibitem{li2020polyhedron}
Chao Li.
\newblock A polyhedron comparison theorem for 3-manifolds with positive scalar
  curvature.
\newblock {\em Inventiones mathematicae}, 219:1--37, 2020.

\bibitem{li2024dihedral}
Chao Li.
\newblock The dihedral rigidity conjecture for $ n $-prisms.
\newblock {\em Journal of Differential Geometry}, 126(1):329--361, 2024.

\bibitem{shi2002positive}
Yuguang Shi and Luen-Fai Tam.
\newblock Positive mass theorem and the boundary behaviors of compact manifolds
  with nonnegative scalar curvature.
\newblock {\em Journal of Differential Geometry}, 62(1):79--125, 2002.

\bibitem{wang2021gromov}
Jinmin Wang, Zhizhang Xie, and Guoliang Yu.
\newblock On gromov's dihedral extremality and rigidity conjectures.
\newblock {\em arXiv preprint arXiv:2112.01510}, 2021.

\end{thebibliography}
\end{document}